\documentclass[5p]{elsarticle}
\usepackage{amsmath}
\usepackage{booktabs}
\usepackage[hidelinks]{hyperref}
\usepackage[table]{xcolor}
\usepackage{amssymb}	
\usepackage{graphicx}
\usepackage{setspace}
\usepackage{algpseudocode}
\usepackage{algorithm}
\usepackage[skip = 5pt]{caption}
\usepackage{tikz}
\usepackage{booktabs}
\usepackage[bottom]{footmisc}
\usetikzlibrary{shapes,arrows, positioning, fit}
\usepackage{epstopdf}
\usepackage{verbatim}
\usepackage{natbib}

\let\Algorithm\algorithm
\renewcommand\algorithm[1][]{\Algorithm[#1]\setstretch{1.3}}

\algnewcommand{\inputs}[1]{%
	\State \textbf{inputs:} \Statex \parbox[t]{.8\linewidth}{\raggedright #1}	
}
\algnewcommand{\initialize}[1]{%
	\State \textbf{initialize:}
	\State \hspace*{\algorithmicindent}\parbox[t]{.8\linewidth}{\raggedright #1}
}

\allowdisplaybreaks[1]

\tikzstyle{block} = [rectangle, draw, fill=lightgray!20, 
text centered, rounded corners, minimum height=2em]
\tikzstyle{blockfix} = [rectangle, draw, fill=lightgray!20, 
text width = 4em, text centered, rounded corners, minimum height=2em]
\tikzstyle{line} = [draw, -latex']
\tikzstyle{cloud} = [draw, ellipse,fill=black!20, node distance=2.5cm,minimum height=2em]

\renewcommand{\vec}[1]{\mathbf{#1}}

\makeatletter
\providecommand{\doi}[1]{%
	\begingroup
	\let\bibinfo\@secondoftwo
	\urlstyle{rm}%
	\href{http://dx.doi.org/#1}{%
		doi:\discretionary{}{}{}%
		\nolinkurl{#1}
	}%
	\endgroup
}
\makeatother
\begin{document}
	
	\begin{frontmatter}
		
		\title{Improved Sparse Low-Rank Matrix Estimation}
		\author[poly]{Ankit Parekh\corref{cor1}}
		\author[polyEE]{Ivan W. Selesnick}
		\cortext[cor1]{Corresponding author. Email address: ankit.parekh@nyu.edu \\ Source Code available at \url{https://github.com/aparek/ISLRMatrix} \\
		}
		\address[poly]{Dept.\ of Mathematics, Tandon School of Engineering, New York University}
		\address[polyEE]{Dept.\ of Electrical and Computer Engineering, Tandon School of Engineering, New York University}

		\begin{abstract}
			We address the problem of estimating a sparse low-rank matrix from its noisy observation. We propose an objective function consisting of a data-fidelity term and two parameterized non-convex penalty functions. Further, we show how to set the parameters of the non-convex penalty functions, in order to ensure that the objective function is strictly convex. The proposed objective function better estimates sparse low-rank matrices than a convex method which utilizes the sum of the nuclear norm and the $\ell_1$ norm. We derive an algorithm (as an instance of ADMM) to solve the proposed problem, and guarantee its convergence provided the scalar augmented Lagrangian parameter is set appropriately. We demonstrate the proposed method for denoising an audio signal and an adjacency matrix representing protein interactions in the `Escherichia coli' bacteria. 
		\end{abstract}
		\begin{keyword}
			Low-rank matrix, sparse matrix, speech denoising, nonconvex regularization, convex optimization
		\end{keyword}
		
	\end{frontmatter}

\section{Introduction}
\label{introduction}
We aim to estimate a sparse low-rank matrix $\vec{X} \in \mathbb{R}^{m \times n}$ from its noisy observation $\vec{Y} \in \mathbb{R}^{m \times n}$, i.e.,	
\begin{align}
	\label{eq::signal model}
	\vec{Y} = \vec{X}+\vec{W}, \qquad \vec{W} \in \mathbb{R}^{m \times n},
\end{align} where $\vec{W}$ represents additive white Gaussian noise (AWGN) matrix. The estimation of sparse low-rank matrices has been studied \cite{Buja2013} and used for various applications such as covariance matrix estimation \cite{Bien2011, ElKaroui2008,Zhou2014a}, subspace clustering \cite{Han2015}, biclustering \cite{Lee2010}, sparse reduced rank regression \cite{Bunea2012,Chen2012}, graph denoising and link prediction \cite{Richard2012,Richard2014}, image classification \cite{Zhang2013} and hyperspectral unmixing \cite{Giampouras2016}. 

In order to estimate the sparse low-rank matrix $\vec{X}$, it has been proposed \cite{Richard2012} to solve the following optimization problem 
\begin{align}
	\label{eq::SLR}
	\arg\min_{\vec{X} \in \mathbb{R}^{m \times n}} \Biggl\lbrace \dfrac{1}{2}\|\vec{Y}-\vec{X}\|_F^2 + \lambda_0\|\vec{X}\|_* + \lambda_1 \|\vec{X}\|_1 \Biggr\rbrace , 
\end{align} where $\|\cdot\|_*$ is the nuclear norm, $\|\cdot \|_1$ is the entry-wise $\ell_1$ norm and $\lambda_i \geqslant 0$ are the regularization parameters. The nuclear norm induces sparsity of the singular values of the matrix $\vec{X}$, while the entry-wise $\ell_1$ norm induces sparsity of the elements of $\vec{X}$. 

The nuclear norm and the $\ell_1$ norm are convex relaxations of the non-convex rank and sparsity constraints, respectively. The nuclear norm can be considered as the $\ell_1$ norm applied to the singular values of the matrix. It is known that the $\ell_1$ norm underestimates non-zero signal values, when used as a sparsity-inducing regularizer. As a result, the sparse low-rank (SLR) problem in \eqref{eq::SLR} can be considered, in general, to be over-relaxed \cite{Jojic2011}. Further, it is known that the performance of nuclear norm for sparse regularization of the singular values is sub-optimal \cite{Nadakuditi2014}. 

In order to estimate the non-zero signal values more accurately, non-convex regularization has been favored over convex regularization \cite{Chartrand_2009_ISBI,Trzasko2009,Repetti2014,Portilla2009,Xin2015}. Furthermore, it has been shown that non-convex penalty functions can induce sparsity of the singular values more effectively than the nuclear norm \cite{Lu2014,Gu2014,Hansson2012,Chartrand2012,Parekh2015sub}. Indeed, it was shown that nonconvex regularizers are better able to estimate simultaneously sparse and low-rank matrices in the context of spectral unmixing for hyperspectral images \cite{Giampouras2016}. The use of non-convex regularizers (penalty functions), however, generally leads to non-convex optimization problems. The non-convex optimization problems suffer from numerous issues (sub-optimal local minima, sensitivity to changes in the input data and the regularization parameters, non-convergence, etc.).

In this paper, we avoid the issues of non-convexity by using parameterized penalty functions, which aid in ensuring the strict convexity of the proposed objective function. We propose to solve the following improved sparse low-rank (ISLR) formulation
\begin{align}
	\label{eq::cost function}
	\arg\min_{\vec{X} \in \mathbb{R}^{m \times n}} \Biggl\lbrace F(\vec{X}):= &\dfrac{1}{2}\|\vec{Y - X}\|_F^2 + \lambda_0\sum_{i=1}^{k}\phi(\sigma_i(\vec{X});a_0) \nonumber \\
	&+ \lambda_1\sum_{i=1}^{m}\sum_{j=1}^{n} \phi(\vec{X}_{i,j};a_1) \Biggr\rbrace,
\end{align} where $k = \min(m,n)$ and 	 $\phi\colon \mathbb{R} \to \mathbb{R}$ is a parameterized non-convex penalty function (see Sec.~\ref{subsection::Penalty Functions}). Note that, if $\lambda_1 = 0$, then the ISLR formulation reduces to the generalized nuclear norm minimization problem \cite{Lu2014b,Parekh2015sub}. Further, if $\lambda_1=0$ and $\phi(x;a) = |x|$, then the ISLR problem \eqref{eq::cost function} reduces to the singular value thresholding (SVT) problem. 

The contributions of this paper are two-fold. First, we show how to set the parameters $a_0$ and $a_1$ to ensure that the function $F$ in \eqref{eq::cost function} is strictly convex. Second, we provide an ADMM based algorithm to solve \eqref{eq::cost function}, which utilizes single variable-splitting compared to two variable-splitting as in \cite{Zhang2013}. We guarantee the convergence of ADMM, provided the scalar augmented Lagrangian parameter $\mu$, satisfies $\mu > 1$. 

\subsection{Related work}

The parameterized non-convex penalty functions used in this paper have designated non-convexity, which enables the overall objective function $F$ in \eqref{eq::cost function} to be strictly convex. In particular, if the parameters $a_0$ and $a_1$ exceed their critical value, then the function $F$ in \eqref{eq::cost function} is non-convex. A similar framework of convex objective functions with non-convex regularization was studied for several signal processing applications (see for eg., \cite{Selesnick2015}, \cite{Ding2015}, \cite{Lanza2015}, \cite{Parekh2015sub2} and the references therein). It was reported that non-convex regularization outperformed convex regularization methods for these applications. 

The sparse low-rank (SLR) formulation in \eqref{eq::SLR} is different from the low-rank + sparse decomposition \cite{Candes2011}, also known as the robust principal component analysis (RPCA). Both the SLR and the RPCA formulations utilize the nuclear norm and the $\ell_1$ norm as sparsity-inducing regularizers \cite{Zhou2014,Zhou2011}. The RPCA formulation aims to estimate the matrix, which is the sum of a low-rank and a sparse matrix. Note that, in the case of RPCA, the matrix to be estimated is itself neither sparse or low-rank \cite{Chandrasekaran2009,Chandrasekaran2009J}. In contrast, the SLR problem \eqref{eq::SLR}, and the one proposed in this paper, considers the case wherein the matrix to be estimated is simultaneously sparse and low-rank (similar to \cite{Giampouras2016}). 

Several well-studied convex optimization algorithms, such as ADMM \cite{Goldstein2009,Afonso2010}, ISTA/FISTA \cite{Beck2009,Figueiredo2007}, and proximal gradient methods \cite{Combettes2011} can be applied to solve convex objective functions of the type \eqref{eq::SLR}. The SLR objective function \eqref{eq::SLR}, has been solved using Generalized Forward-Backward \cite{Raguet2013}, Incremental Proximal Descent \cite{Richard2012} (introduced in \cite{Bertsekas2010}), Majorization-Minimization \cite{Hu2012}, and the Inexact Augmented Lagrangian Multiplier (IALM) method \cite{Lin2010}. The IALM method can also be used to solve the SLR problem, although with a different data-fidelity term \cite{Zhang2013}.

\section{Preliminaries}
\label{section::Preliminaries}
We denote vectors and matrices by lower and upper case letters respectively. For a matrix $\vec{Y}$, we use the following entry-wise norms, 
\begin{align}
\label{eq::Entry-wise Norms}
\|\vec{Y}\|_F^2 := \sum_{i,j} |\vec{Y}_{i,j}|^2, \quad 	\|\vec{Y}\|_1 := \sum_{i,j}|\vec{Y}_{i,j}|.
\end{align} Further, we use the nuclear norm (also called the `Schatten-1' norm) defined as
\begin{align}
\label{eq::Nuclear Norm}
\|\vec{Y}\|_* := \sum_{i=1}^{k} \sigma_i(\vec{Y}), 
\end{align} where $\sigma_i(\vec{Y})$ represent the singular values of the matrix $\vec{Y} \in \mathbb{R}^{m \times n}$ and $k = \min (m,n)$.

\subsection{Parameterized non-convex penalty functions}
\label{subsection::Penalty Functions}
We propose to use non-convex penalty functions $\phi(x;a)$ parameterized by the parameter $a \geqslant 0$. The value of $a$ provides the degree of non-convexity of the penalty functions. Below we define such non-convex penalty functions and list their properties. 
\newtheorem{assumption}{\bf Assumption}
\begin{assumption}
	\label{theorem::assumption 1}
	The non-convex penalty function $\phi\colon\mathbb{R} \to \mathbb{R}$ satisfies the following
	\begin{enumerate}
		\item $\phi$ is continuous on $\mathbb{R}$, twice differentiable on $\mathbb{R}\!\setminus\! \lbrace 0\rbrace$ and symmetric, i.e., $\phi(-x; a) = \phi(x; a)$
		\item $\phi'(x) > 0, x > 0$	
		\item $\phi''(x) \leq 0,  x > 0$
		\item $\phi'(0^{+}) = 1$
		\item $\inf\limits_{x\neq0}\phi''(x;a) = \phi''(0^+;a) = -a$
	\end{enumerate}
\end{assumption} 

\begin{figure}
	\centering
	\includegraphics[]{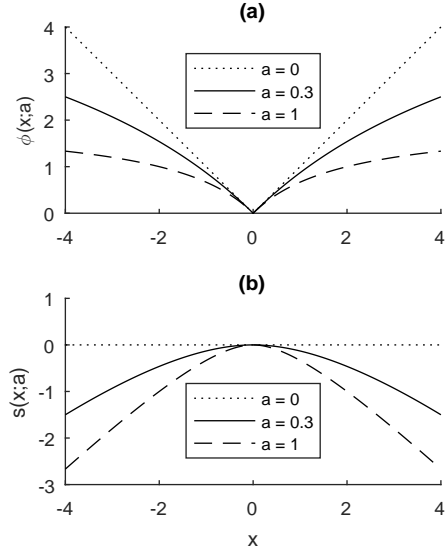}
	\caption{(a) Non-convex penalty function $\phi$ in \eqref{eq::rat penalty} for three values of $a$. (b) The twice continuously differentiable concave function $s(x;a) = \phi(x;a) - |x|$ in \eqref{eq::s function} for the corresponding values of $a$.}
	\label{fig::penalty functions}
\end{figure}

An example of a non-convex penalty function satisfying Assumption 1 is the rational penalty function \cite{Geman1992} defined as
\begin{align}
\label{eq::rat penalty}
\phi(x;a) := \dfrac{|x|}{1 + a|x|/2}, \qquad a \geqslant 0.
\end{align} 
The $\ell_1$ norm is recovered as a special case of the non-convex rational penalty function (i.e., if $a = 0$, then $\phi(x;0) = |x|$). Figure~\ref{fig::penalty functions}(a) shows the rational penalty function \eqref{eq::rat penalty} for different values of $a$. Other examples of penalty functions satisfying Assumption \ref{theorem::assumption 1} are the logarithmic penalty \cite{Candex_2008_JFAP,Mohan2012}, arctangent penalty \cite{Selesnick_2014_MSC} and the Laplace penalty \cite{Trzasko2009}.

The proximity operator of $\phi$ \cite{Combettes2007}, $\mbox{prox}_{\phi}:\mathbb{R}\to\mathbb{R}$, is defined as
\begin{align}
\mbox{prox}_{\phi}(y; \lambda,a) := \arg\min_{x \in \mathbb{R}} \left\lbrace \frac{1}{2}(y-x)^2 + \lambda\phi(x;a) \right\rbrace. \nonumber
\end{align} The proximity operator associated with the function $\phi(x;a)$, satisfying Assumption 1, is continuous with 
\begin{align}
\label{theta}
\mbox{prox}_{\phi}(y; \lambda,a) = 0, \forall |y| < \lambda,
\end{align} if $0 \leqslant a < 1/\lambda$. The proximity operators associated with the arctangent and the logarithmic penalty are provided in \cite{Selesnick_2014_MSC}. Note that for $a=0$, the proximity operator is the soft-threshold function \cite{Donoho1995}.

The proximity operator associated with the $\ell_1$ norm is the well-known soft-threshold function \cite{Donoho1995}. Note that the soft-threshold function underestimates non-zero values. In contrast, the proximity operators, associated with the non-convex penalty functions satisfying Assumption 1, approach the identity function asymptotically \cite{Selesnick_2014_MSC}. Thus, the proximity operators used in this paper estimate the non-zero values more accurately than the $\ell_1$ norm. 

\section{Convexity Condition}
\label{section::Convexity condition}
In this section we derive a condition to ensure that the function $F$ in \eqref{eq::cost function} is strictly convex. In particular, we show that the objective function $F$ is strictly convex if $a_0$ and $a_1$ lie inside a designated region. To this end, we note the following lemmas. 
\newtheorem{lemma}{\bf Lemma}
\begin{lemma}{\cite{Parekh2015}}
	\label{lemma 1}
	Let $\phi\colon\mathbb{R}\to\mathbb{R}$ be a non-convex penalty function satisfying Assumption \ref{theorem::assumption 1}. The function $s\colon\mathbb{R} \to \mathbb{R}$ defined as
	\begin{align}
	\label{eq::s function}
	s(x;a) := \phi(x;a) - |x|,
	\end{align} is twice continuously differentiable, concave and
	\begin{align}
	-a \leqslant s''(x;a) \leqslant 0. 
	\end{align}
\end{lemma} The twice continuously differentiable function $s(x;a) = \phi(x;a) - |x|$ is shown in Fig.~\ref{fig::penalty functions}(b), for three values of $a$.

\begin{lemma}{\cite{Parekh2015sub}}
	\label{Lemma 2}
	Let $\phi\colon\mathbb{R} \to \mathbb{R}$ be a non-convex penalty function satisfying Assumption 1 and $s\colon\mathbb{R}\to\mathbb{R}$ be the function as defined in Lemma 1. The function $G_1\colon\mathbb{R}^{m \times n}\to\mathbb{R}$ defined as
	\begin{align}
	\label{eq::ELMA cost convex}
	G_1(\vec{X}) := \dfrac{\alpha_1}{2} \|\vec{Y}- \vec{X}\|_F^2 + \lambda_0 \sum_{i=1}^{k}s\bigl(\sigma_i(\vec{X}); a_0\bigr),
	\end{align} where $k = \min(m,n)$ and $\alpha_1 > 0$, is strictly convex if
	\begin{align}
	\label{eq::first convexity condition}
	0 \leqslant a_0 < \dfrac{\alpha_1}{\lambda_0}.
	\end{align}
\end{lemma} Note that the proof of Lemma \ref{Lemma 2} in \cite{Parekh2015sub} considers the case if $\alpha_1 = 1$, however the generalization for $\alpha_1 > 0$ follows directly. 

\begin{lemma}
	\label{Lemma 3}
	Let $\phi\colon\mathbb{R} \to \mathbb{R}$ be a non-convex penalty function satisfying Assumption 1 and $s\colon\mathbb{R}\to\mathbb{R}$ be the function as defined in Lemma 1. The function $G_2\colon\mathbb{R}^{m \times n}\to\mathbb{R}$ defined as
	\begin{align}
	\label{eq::second cost function with l1 norm}
	G_2(\vec{X}) := \dfrac{\alpha_2}{2}\|\vec{Y}-\vec{X}\|_F^2 + \lambda_1 \sum_{i=1}^{m}\sum_{j=1}^{n}s(\vec{X}_{ij};a_1),
	\end{align}where $\alpha_2 > 0$, is strictly convex if
	\begin{align}
	\label{eq::second convexity condition}
	0 \leqslant a_1 < \dfrac{\alpha_2}{\lambda_1}.
	\end{align}
\end{lemma}
\newtheorem{proof}{Proof}
\begin{proof}
	The Frobenius norm of a matrix $\vec{Y}\in\mathbb{R}^{m\times n}$ can be viewed as the $\ell_2$ norm of a vector $y \in \mathbb{R}^{mn}$, where $y$ contains the entries of the matrix $\vec{Y}$. Similarly, we stack the entries of the matrix $\vec{X} \in \mathbb{R}^{m \times n}$ into a vector $\vec{x} \in \mathbb{R}^{mn}$ and re-write the function $G_2$ as
	\begin{align}
	G_2(\vec{x}) = \dfrac{\alpha_2}{2}\|\vec{y-x}\|_2^2 + \lambda_1 \sum_{i=1}^{mn}s(x_i,a_1). 
	\end{align} In order to ensure the strict convexity of $G_2$, we seek to ensure that the Hessian of $G_2$ be positive definite (i.e., $\nabla^2 G \succ 0$). To this end, the Hessian of $G_2$ is given by
	\begin{align}
	\label{eq:: Hessian of G2}
	\nabla^2 G_2 = \alpha_2 \vec{I} + \lambda_1 \cdot \mathrm{diag}\bigl(s''(x_1;a_1),\hdots , s''(x_{mn};a_1) \bigr),
	\end{align} where $\mathrm{diag}(\cdot)$ represents a diagonal matrix. Note that the identity matrix in \eqref{eq:: Hessian of G2} is of size $mn \times mn$. To ensure that $\nabla^2 G_2$ is positive definite, we seek to ensure 
	\begin{align}
	\alpha_2 + \lambda_2 s''(t; a_1) &> 0, \qquad \forall t \in \mathbb{R}, \\
	\label{eq::lemma 2 inequality}
	s''(t; a_1)  &> -\dfrac{\alpha_2}{\lambda_1}.
	\end{align} Thus, using Lemma 1 and \eqref{eq::lemma 2 inequality}, the Hessian of $G_2$, i.e., $\nabla^2 G_2$, is positive definite if $0 \leqslant a_1 < \alpha_2/\lambda_1$.
\end{proof} 

The following theorem provides the critical values of the parameters $a_0$ and $a_1$ to ensure that the function $F$ in \eqref{eq::cost function} is strictly convex. 
\newtheorem{theorem}{\bf Theorem}
\begin{theorem}
	\label{theorem::Convexity Condition}
	Let $\phi\colon\mathbb{R}\to\mathbb{R}$ be a parameterized non-convex penalty function satisfying Assumption \ref{theorem::assumption 1}. The function $F\colon\mathbb{R}^{m \times n}\to\mathbb{R}$ defined as
	\begin{align}
	F(\vec{X}):= \dfrac{1}{2}\|\vec{Y - X}\|_F^2 &+ \lambda_0\sum_{i=1}^{m}\phi(\sigma_i(\vec{X});a_0) \nonumber \\
	&+ \lambda_1\sum_{i=1}^{m}\sum_{j=1}^{n} \phi(\vec{X}_{i,j};a_1),
	\end{align} is strictly convex if 
	\begin{align}
	\label{eq::convexity condition}
	0 \leqslant a_0\lambda_0 + a_1\lambda_1 < 1.
	\end{align}
\end{theorem}

\begin{proof}
	Let $\alpha \in [0,1]$. Consider the function $G\colon \mathbb{R}^{m \times n}\to \mathbb{R}$ defined as
	\begin{align}
	\label{eq::G function}
	G(\vec{X}) :&= \dfrac{\alpha}{2}\|\vec{Y} - \vec{X}\|_F^2 + \lambda_0 \sum_{i=1}^{m}s\bigl(\sigma_i(\vec{X});a_0\bigr) \nonumber \\ 
	&+ \dfrac{1-\alpha}{2}\|\vec{Y} - \vec{X}\|_F^2 +  \lambda_1 \sum_{i=1}^{m}\sum_{j=1}^{n} s(\vec{X}_{ij};a_1),
	\end{align} where $s\colon\mathbb{R}\to\mathbb{R}$ is defined in Lemma 1. Using the functions $G_1$ and $G_2$, as defined in \eqref{eq::ELMA cost convex} and \eqref{eq::second cost function with l1 norm} respectively, the function $G$ can be written as
	\begin{align}
	G(\vec{X}) = G_1(\vec{X}) + G_2(\vec{X}), 
	\end{align} where $\alpha_1 = \alpha$, and $\alpha_2 = 1-\alpha$. Due to Lemma \ref{Lemma 2} and Lemma \ref{Lemma 3}, the function $G$ is strictly convex (being a sum of two strictly convex functions) if $a_0$ and $a_1$ satisfy the following inequalities,
	\begin{align}
	\label{eq::inequality 1}
	&0 \leqslant a_0 < \dfrac{\alpha}{\lambda_0}, \\
	\label{eq::inequality 2}
	&0 \leqslant a_1 < \dfrac{1-\alpha}{\lambda_1}.
	\intertext{Combining \eqref{eq::inequality 1} and \eqref{eq::inequality 2}, we obtain,}
	\label{eq::final inequality}
	&0 \leqslant a_0\lambda_0 + a_1\lambda_1 < 1.
	\end{align} As a result, the function $G$ is strictly convex if $a_0$ and $a_1$ satisfy the inequality in \eqref{eq::final inequality}. Recall that $\phi(x;a) = s(x;a) + |x|$ from \eqref{eq::s function}, due to which the function $F$ in \eqref{eq::cost function} can be written as
	\begin{align}
	F(\vec{X}) = &\dfrac{1}{2}\|\vec{Y - X}\|_F^2 + \lambda_0\sum_{i=1}^{m} \Bigl[ s(\sigma_i(\vec{X});a_0) + |\sigma_i(\vec{X})| \Bigr]\nonumber \\
	& \qquad + \lambda_1\sum_{i=1}^{m}\sum_{j=1}^{n} \Bigl[ s(\vec{X}_{i,j};a_1) + |\vec{X}_{i,j}| \Bigr] \\
	= &\dfrac{1}{2}\|\vec{Y - X}\|_F^2 + \lambda_0 \sum_{i=1}^{m} s(\sigma_i(\vec{X});a_0) + \|\vec{X}\|_* \nonumber \\
	& \qquad + \lambda_1\sum_{i=1}^{m}\sum_{j=1}^{n} s(\vec{X}_{i,j};a_1) + \|\vec{X}\|_1 \\
	= & G(\vec{X}) + \|\vec{X}\|_* + \|\vec{X}\|_1.
	\end{align} Thus, the function $F$ is strictly convex (being a sum of a strictly convex function and convex functions) if $a_0$ and $a_1$ satisfy the inequality \eqref{eq::final inequality}.
\end{proof}

\begin{figure}
	\centering
	\includegraphics[]{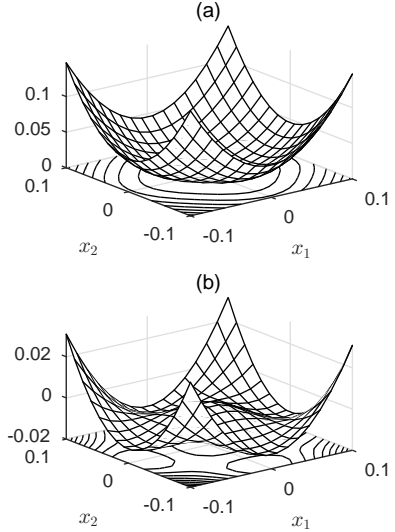}
	\caption{Illustration of the convexity condition provided by Theorem 1. (a) The function $G$ is convex when $a_0$ and $a_1$ satisfy the inequality \eqref{eq::convexity condition}. (b) The function $G$ is non-convex otherwise (multiple local minima can be seen in the contour plot).}
	\label{fig::Convexity condition}
\end{figure}

The convexity condition provided by Theorem 1 is illustrated in Fig.~\ref{fig::Convexity condition}. The matrix $\vec{X}$ is constructed by tiling 10 copies of the matrix $\vec{Z} \in \mathbb{R}^{2 \times 2}$,
\begin{align}
Z = \left[ \begin{array}{cc}
x_1 & x_2 \\
x_2 & x_1
\end{array}\right], \quad x_i \in \mathbb{R},
\end{align} and randomly setting 70\% of its entries zero. Thus, the matrix $\vec{X}$ is of rank 2. We use $\lambda_0 = \lambda_1 = 1$ and set $\vec{Y} = 0$. We set the value of $a_0 = 0.8$, and $a_1 = 0.19$, as per \eqref{eq::convexity condition}, to ensure that the function $G$ in \eqref{eq::G function} is strictly convex. As seen in Fig.~\ref{fig::Convexity condition}(a), the function $G$ is strictly convex. However, on increasing the value of $a_1$ to $1$, the function $G$ is non-convex, as seen in Fig.~\ref{fig::Convexity condition}(b). 

\begin{figure}[t!]
	\centering
	\includegraphics[]{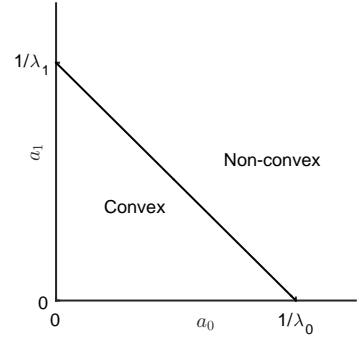}
	\caption{The function $F$ in \eqref{eq::cost function} is strictly convex for all values of $a_0$ and $a_1$ inside the triangular region.}
	\label{fig::convexity region}
\end{figure}

The inequality \eqref{eq::convexity condition} given by Theorem 1, constitutes a convexity triangle for the function $F$ in \eqref{eq::cost function}. For all values of $a_0$ and $a_1$ inside the triangular region of Fig.~\ref{fig::convexity region} the function $F$ is strictly convex. However, the function $F$ in \eqref{eq::cost function} is non-convex for values of $a_0$ and $a_1$ outside the triangular region.

\section{Algorithm}
\label{section::Algorithm}
We use the alternating direction method of multipliers (ADMM) \cite{Boyd2010} in conjunction with variable splitting to derive an algorithm for the solution to the ISLR problem \eqref{eq::cost function}. The convergence of ADMM to the global minimum is guaranteed when the objective function is a sum of two convex functions \cite{Eckstein1992}. The convergence of ADMM for a non-convex optimization problem to a stationary point is guaranteed under certain mild assumptions \cite{Hong2015}. 

The following theorem derives an algorithm for solving the ISLR problem and guarantees it convergence. In particular, the theorem provides a condition on the value of the scalar augmented Lagrangian parameter $\mu$, to ensure that the sub-problems of ADMM are strictly convex.

\begin{algorithm}[tb]
\caption{Solution to the proposed ISLR problem. The objective function $F$ is given in \eqref{eq::cost function}.}
\label{table::Algorithm}
\begin{algorithmic}[1]
	\State {\bf Input:} $\vec{Y}$, $\lambda_i$, $a_i$, $\mu$, $\epsilon$ 
	\State {\bf Initialize:}  $\vec{Z} = 0$, $\vec{D} = 0$ 
	\Repeat
	\State $\vec{X} \gets \mbox{prox}_{\phi} \left(\dfrac{1}{1+\mu} \bigl(\vec{Y} + \mu(\vec{Z} + \vec{D})\bigr) ;\dfrac{\lambda_1}{1 + \mu},a_1 \right)$ 
	\State $[\vec{U}, \vec{\Sigma}, \vec{V}] \gets \mbox{SVD} \left(\vec{X} - \vec{D} \right)$ 
	\State $\vec{Z} \gets \vec{U} \cdot \mbox{prox}_{\phi} (\vec{\Sigma};\lambda_0/\mu, a_0) \cdot \vec{V}^T $
	\State $\vec{D} \gets \vec{D} - (\vec{X} - \vec{Z})$
	\Until $\|F(\vec{X}^{k}) - F(\vec{X}^{k-1})\|_2 < \epsilon \|F(\vec{X}^k)\|_2$
\end{algorithmic}
\end{algorithm}

\begin{theorem}
	\label{theorem 2}
	Let $\phi\colon\mathbb{R} \to \mathbb{R}$ be a non-convex penalty function satisfying Assumption 1. Let $a_0$ and $a_1$ satisfy
	\begin{align}
	0 \leqslant a_0\lambda_0 + a_1\lambda_1 < 1,
	\end{align}for $\lambda_0, \lambda_1 \geq 0$. Further, let $\mu$ be the scalar augmented Lagrangian parameter. If $\mu > 1$, then the iterative algorithm in Table~\ref{table::Algorithm} converges to the global minimum of the function $F$ in \eqref{eq::cost function}.	
\end{theorem}

\begin{proof}
	Without loss of generality, we set $m = n$. We re-write the ISLR objective function \eqref{eq::cost function} using variable splitting \cite{Afonso2010} as
		\begin{align}
		\arg\min_{\vec{X}} &\Biggl\lbrace \dfrac{1}{2} \|\vec{Y}-\vec{X}\|_F^2 + \lambda_0\sum_{i=1}^{m}\phi(\sigma_i(\vec{\vec{Z}});a_0) \nonumber \\
		&\qquad\quad + \lambda_1\sum_{i=1}^{m}\sum_{j=1}^{n} \phi(\vec{X}_{i,j};a_1) \Biggr\rbrace, \nonumber \\
		\label{eq::Variable splitting formulation}
		\mbox{s.t.} \quad &\vec{X} = \vec{Z}.
		\end{align} The minimization of the ISLR objective function in \eqref{eq::Variable splitting formulation} is separable in $\vec{X}$ and $\vec{U}$. Applying ADMM to \eqref{eq::Variable splitting formulation}, yields the following iterative procedure, where $\mu$ is the scalar augmented Lagrangian parameter and $\vec{D} \in \mathbb{R}^{m \times n}$ is update variable.
	\begin{subequations}
		\label{eq::sub-problems}
		\begin{align}
		\label{eq::sub-problem for X}
		\vec{X} &\gets \arg\min_{\vec{X}} \Biggl\lbrace \dfrac{1}{2}\|\vec{Y} - \vec{X}\|_F^2 + \dfrac{\mu}{2}\|\vec{X}-(\vec{Z} + \vec{D})\|_F^2 \nonumber \\
		&\qquad \qquad \qquad \qquad + \lambda_1 \sum_{i=1}^{m}\sum_{j=1}^{n} \phi(\vec{X}_{i,j}; a_1) \Biggr\rbrace \\
		\label{eq::sub-problem for U}
		\vec{Z} &\gets \arg\min_{\vec{Z}} \Biggl\lbrace \dfrac{\mu}{2}\|\vec{X} - \vec{D} - \vec{Z}\|_F^2 + \lambda_0 \sum_{i=1}^{m}\phi\bigl(\sigma_i(\vec{Z});a_0\bigr) \Biggr\rbrace \\
		\label{eq::updated multiplier}
		\vec{D} &\gets \vec{D} - (\vec{X} - \vec{Z})
		\end{align}
	\end{subequations} 
	Combining the quadratic terms and ignoring the constant terms, the sub-problem \eqref{eq::sub-problem for X} can be written as
	\begin{align}
	\label{eq::Simplified subproblem for X}
	\vec{X} &\gets \arg\min_{\vec{X}} \Biggl\lbrace \dfrac{1}{2}\biggl\|\dfrac{1}{1+\mu} \left(\vec{Y} + \mu(\vec{Z} + \vec{D})\right) - \vec{X}\biggr\|_F^2 \nonumber \\
	& \qquad \qquad \qquad + \dfrac{\lambda_1}{1+\mu}\sum_{i=1}^{m}\sum_{j=1}^{n}\phi(\vec{X}_{i,j};a_1) \Biggr\rbrace.
	\end{align}
	Since $a_1 < 1/\lambda_1$, as per the assumption, the sub-problem \eqref{eq::Simplified subproblem for X} is strictly convex and its solution can be obtained using the proximal operator associated with $\phi$, i.e., 
	\begin{align}
	\label{eq::solution to X}
	\vec{X} &\gets \mathrm{prox}_{\phi}(\tilde{\vec{Y}}; \lambda_1 / (1+\mu), a_1 ).
	\end{align} 
	The sub-problem \eqref{eq::sub-problem for U} is the generalized nuclear norm minimization problem, whose solution in closed form is provided by Theorem 2 in \cite{Parekh2015sub}. Note that the solution is guaranteed to be the global minimum if $a_0 < \mu/\lambda_0$. Hence, using the inequality \eqref{eq::convexity condition}, the sub-problem \eqref{eq::sub-problem for U} is guaranteed to be strictly convex if $\mu > 1$. As a result, using $(\vec{X} - \vec{D}) = \vec{U} \vec{\Sigma} \vec{V}^T$ as the singular value decomposition (SVD) of the matrix $\vec{X}-\vec{D}$, we get
	\begin{align}
	\label{eq::solution to U}
	\vec{Z} &\gets \vec{U} \cdot \mathrm{prox}_{\phi} (\vec{\Sigma}; \lambda_0/\mu,a_0) \cdot \vec{V}^T.
	\end{align}
	Combining \eqref{eq::solution to X} and \eqref{eq::solution to U}, we obtain the iterative algorithm~\ref{table::Algorithm}, which converges to the global minimum of the ISLR objective function in \eqref{eq::cost function}. 
\end{proof}

\section{Examples}
\label{section::Examples}

We illustrate the proposed ISLR method for estimating simultaneously sparse and low-rank matrices via the following examples. We first describe setting the parameters for the proposed method and then showcase the examples. 

\subsection{Parameter tuning}
\begin{figure}[t!]
	\centering
	\includegraphics[]{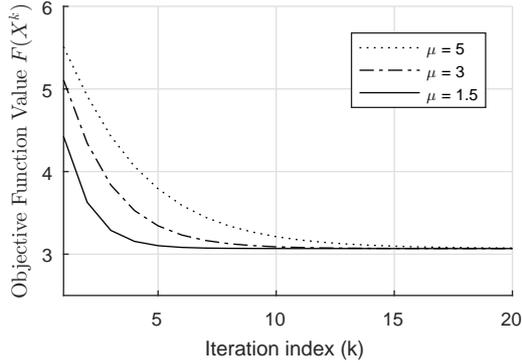}
	\caption{The value of the objective function $F$ in \eqref{eq::cost function} at every iteration of the ISLR algorithm for the speech signal denoising problem.}
	\label{fig::cost function history}
\end{figure} 
The proposed ISLR algorithm~\ref{table::Algorithm} requires the specification of two regularization parameters $\lambda_0$ and $\lambda_1$, two penalty parameters $a_0$ and $a_1$ and the scalar augmented Lagrangian parameter $\mu$. We set the regularization parameters $\lambda_i$ as 
\begin{align}
\lambda_i = \beta_i\sigma, \qquad i = 0,1.,
\end{align} where $\sigma$ is the standard deviation of AWGN (or an estimate) and $\beta_i$ are chosen so as to maximize the signal-to-noise ratio (SNR) for the SLR and the ISLR methods. The values of $a_0$ and $a_1$ are set as
\begin{align}
\label{eq::a_0 and a_1 value}
a_0 &= \dfrac{c}{\lambda_0}, \qquad c \in (0,1) \\
a_1 &= \dfrac{1-a_0\lambda_0}{\lambda_1},
\end{align} respectively. The value of $a_0 \in [0, 1/\lambda_0)$ affects the sparsity of the singular values, and the value of $a_1 \in [0,1/\lambda_1)$ affects the sparsity of the elements of the matrix to be estimated. Thus, if the sparsity of the singular values is favored over the sparsity of the elements of the matrix to be estimated, the point $(a_0,a_1)$ may be set in the lower-right region of the convexity triangle shown in Fig.~\ref{fig::convexity region}. Alternatively, if the sparsity of the elements is preferred over the sparsity of the singular values of the matrix to be estimated, the point $(a_0,a_1)$ may be set in the upper-left region of convexity triangle in Fig.~\ref{fig::convexity region}.

\begin{figure}[t!]
	\centering
	\includegraphics[]{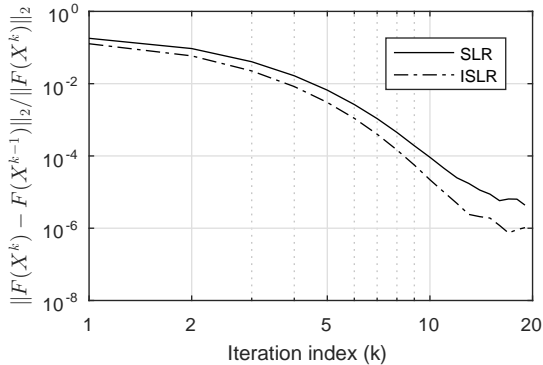}
	\caption{Log-log plot showing the convergence of the ISLR objective function \eqref{eq::cost function} and the SLR objective function \eqref{eq::SLR}.}
	\label{fig::loglog plot}
\end{figure} 

We set the value of $\mu$ as $\mu = 1.5$. As per Theorem 2, the ISLR algorithm listed in Table 1 is guaranteed to converge to the global minimum for $\mu > 1$. However, depending on the value of $\mu$, the convergence may be slow. As such for this example, and the one that follows, we run the ISLR algorithm till a certain tolerance level is reached, i.e., we run the algorithm till
\begin{align}
\label{eq::Tolerance level}
\|F(\vec{X}^{k}) - F(\vec{X}^{k-1})\|_2 < \epsilon \|F(\vec{X}^k)\|_2,
\end{align} where $\epsilon$ is a user-defined tolerance level, usually set to $\epsilon = 10^{-5}$. The value of the objective function \eqref{eq::cost function} for 20 iterations of the ISLR algorithm, with different values of $\mu$, is shown in Fig.~\ref{fig::cost function history}. Fig.~\ref{fig::loglog plot} shows the log-log plot of the stopping criteria for the ISLR objective function and the SLR objective function \eqref{eq::SLR}.  Note that for the examples that follow, we use the arctangent penalty \cite{Selesnick_2014_MSC} as the nonconvex penalty for the proposed ISLR method. 

\subsection{Synthetic data}

We generate a synthetic matrix $\vec{M} \in \mathbb{R}^{m \times n}$ of rank $k$ using two random matrices $\vec{A} \in \mathbb{R}^{m \times k}$ and $\vec{B}\in \mathbb{R}^{k \times n}$ such that
\begin{align}
\label{eq::M matrix}
\vec{M} := \vec{A}\cdot \vec{B}, 
\end{align} where the entries of $\vec{A}$ and $\vec{B}$ are chosen from an i.i.d standard normal distribution. To measure the performance of the proposed ISLR method and the SLR method, we use the normalized root square error (RSE) defined as
\begin{align}
\label{eq::RSE}
\mathrm{RSE} := \dfrac{\|\vec{X}_{\mathrm{est}} - \vec{X}_{\mathrm{org}}\|_F}{\|\vec{X}_{\mathrm{org}}\|_F}, 
\end{align} where $\vec{X}_{\mathrm{est}}$ represents the estimated matrix and $\vec{X}_{\mathrm{org}}$ represents the desired clean matrix. 

\begin{figure}
	\centering
	\includegraphics[]{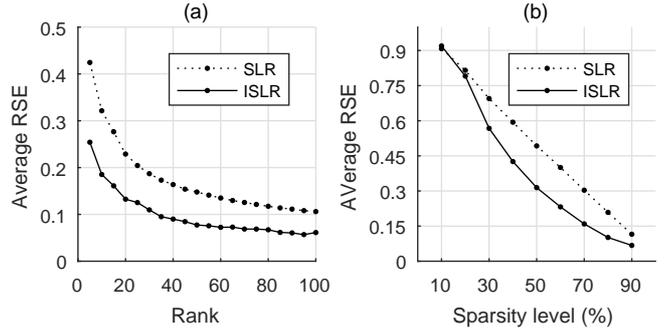}
	\caption{(a) Average RSE as a function of the rank of the input matrix. (b) Average RSE as a function of the level of sparsity of the input matrix. }
	\label{Fig::Average RSE as a function}
\end{figure}

We consider two types of simulations: RSE as a function of the rank ($k$) of the synthetically generated matrix and RSE as a function of the sparsity level of the input matrix. For the first simulation, we fix the sparsity level at 60\% (i.e., approximately 40\% of the entries of the clean input matrix $\vec{M}$ are zero) while varying the rank $k$ of the input matrix. We add white Gaussian noise ($\sigma = 0.2$) to $\vec{M}$ to generate a noisy input matrix $\vec{Y}$. We generate 15 matrices for each value of the rank $k$ where $1 \leqslant k \leqslant 100$ in increments of $5$ and denoise them using the proposed ISLR method \eqref{eq::cost function} and the SLR method \eqref{eq::SLR}. For the second simulation, wherein we consider the RSE as a function of the level of sparsity of the input matrix, we synthetically generate 15 matrices $\vec{M}$ of fixed rank $k = 10$ but with varying levels of sparsity (from 10\% to 90\%). Again, we add white Gaussian noise ($\sigma = 0.2$) to generate the noisy matrix $\vec{Y}$.

Figure~\ref{Fig::Average RSE as a function}(a) shows the average RSE values as a function of the rank $k$ of the input matrix. Figure~\ref{Fig::Average RSE as a function}(b) shows the average RSE values as a function of the level of sparsity of the input matrix. The proposed ISLR method consistently obtains lower RSE values than the SLR method. As expected, for both the methods, the proposed ISLR method and the SLR method, RSE values are lower for matrices that are relatively more sparse. On the other hand, as seen in Fig.~\ref{Fig::Average RSE as a function}(a), we observe that for matrices that are not necessarily low-rank but have decaying singular values, lower values of RSE are obtained for both the methods. Note that for both the  simulations, we do a grid search over a range of values for $\beta_0$ and $\beta_1$ to obtain the values of $\lambda_0$ and $\lambda_1$ respectively, which yield the lowest RSE values (recall that $\lambda_i = \beta_i\sigma$, for $i = 0,1$). Furthermore, in both the simulations, we fix the value of $c$ at $c = 0.5$ for setting the values of $a_0$ and $a_1$. 
\subsection{Speech signal denoising}

\begin{figure}[t!]
	\centering
	\includegraphics[]{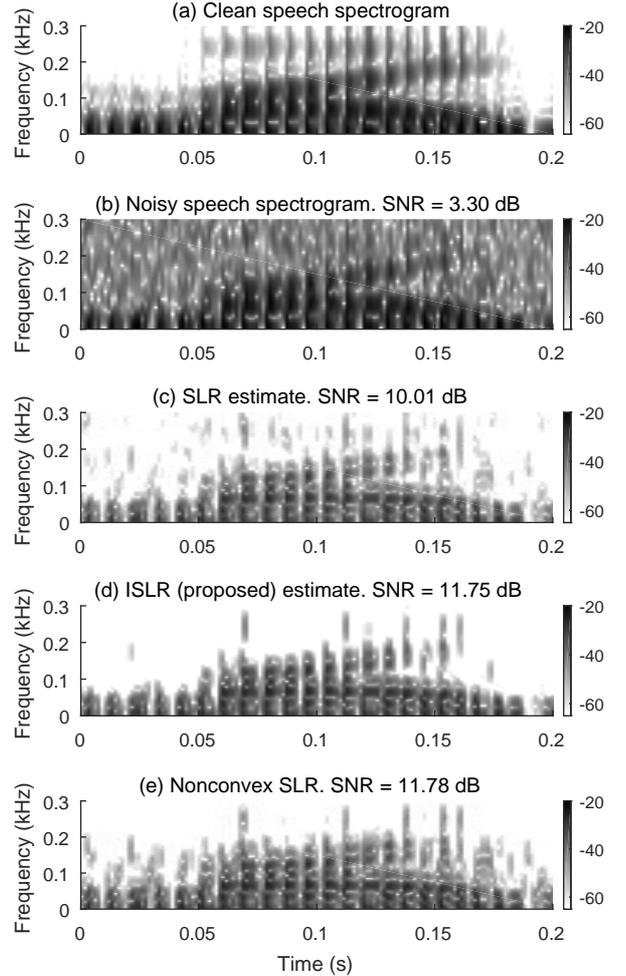}
	\caption{Illustration of proposed method for denoising an audio signal. The colorbar values are in dB.}
	\label{fig::Speech signal denoising}
\end{figure}

We consider the problem of denoising a speech signal with AWGN. We apply the sparse low-rank matrix estimation methods to the spectrogram of the noisy speech signal, and invert the estimated spectrogram to the time domain to obtain the denoised speech signal. Specifically, if $\vec{y}$ is the input speech signal, then the denoised estimate $\vec{x}^*$ is obtained using ISLR as
\begin{align}
\label{eq::Speech estimate}
\vec{x}^* = \vec{S}^{\dagger} \Bigl\lbrace \mbox{ISLR}\bigl(\vec{S(y)};\lambda_i,a_i\bigr) \Bigr\rbrace, \quad i = 0,1.,
\end{align} where $\vec{S}$ and $\vec{S}^{\dagger}$ represent the short-time Fourier transform (STFT) and its inverse, respectively. For this example, we set $\vec{S}$ to be an over-complete STFT, implemented with perfect reconstruction, i.e., $\vec{S}^\dagger\vec{S} = \vec{I}$. The STFT is implemented with 50\% overlap between the windows, for a window size of 64 samples. The DFT length is set to 512 samples. We add AWGN $(\sigma = 0.03)$ to realize the noisy speech signal. We compare the proposed ISLR method with the SLR method \eqref{eq::SLR} and the nonconvex sparse low-rank matrix estimation method \cite{Giampouras2016} which uses the weighted nuclear norm \cite{Gu2014} and the weighted $\ell_1$ norm \cite{Candex_2008_JFAP}.

The spectrogram of the clean speech signal is shown in Fig.~\ref{fig::Speech signal denoising}(a). It can be seen that the spectrogram consists of repeated ridges.  The spectrogram of the noisy signal is shown in Fig.~\ref{fig::Speech signal denoising}(b). Figure~\ref{fig::Speech signal denoising}(c) shows the denoised spectrogram obtained using the SLR method \eqref{eq::SLR} and Fig.~\ref{fig::Speech signal denoising}(d) shows the denoised spectrogram using the ISLR method \eqref{eq::cost function}. Figure.~\ref{fig::Speech signal denoising}(e) shows the spectrogram estimated using the nonconvex method with weighted $\ell_1$ norm and the weighted nuclear norm (see Algorithm 1 in \cite{Giampouras2016}). The ISLR estimated spectrogram has a higher SNR than the SLR estimated spectrogram and contains fewer artifacts. The nonconvex method obtains a slightly higher SNR than the proposed ISLR method and also contains relatively fewer artifacts than the SLR method. The proposed ISLR method obtains SNR values comparable to the state-of-the-art nonconvex method, while being able to guarantee convergence to the unique global minimum. 

\begin{figure}[t!]
	\centering
	\includegraphics[]{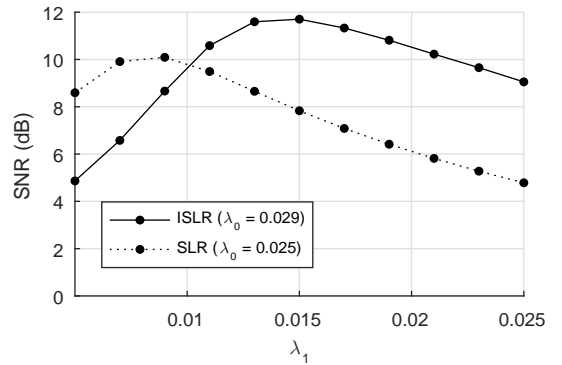}
	\caption{SNR as a function of $\lambda_0$ with fixed $\lambda_1$ for the ISLR and the SLR methods.}
	\label{fig::SNR vs lam}
\end{figure}

Figure~\ref{fig::SNR vs lam} shows the SNR as a function of the regularization parameter $\lambda_1$, when $\lambda_0$ is fixed, for the SLR and the ISLR methods. Note that the standard deviation $\sigma$ of the noise level is also fixed at $(\sigma = 0.03)$. The improvement in the SNR value, when using the ISLR method, is the same when the value of $\lambda_0$ is also varied, in addition to the value of $\lambda_1$. The SNR values are obtained by averaging over 15 realizations for each $\lambda_0,\lambda_1$ pair. 

\subsection{Protein Interactions}
Protein interactions in the `Escherichia coli' bacteria, scored by strength in $[0,2]$, were studied in \cite{Hu2009}. The data can be represented as a weighted graph, which is sparse and low-rank by nature \cite{Richard2012}. The rationale behind the low-rank property of the weighted graph, is that the interactions between two sets of proteins are governed by a small set of factors \cite{Richard2012,Bock2001}. 

\begin{figure}[t!]
	\centering
	\includegraphics[]{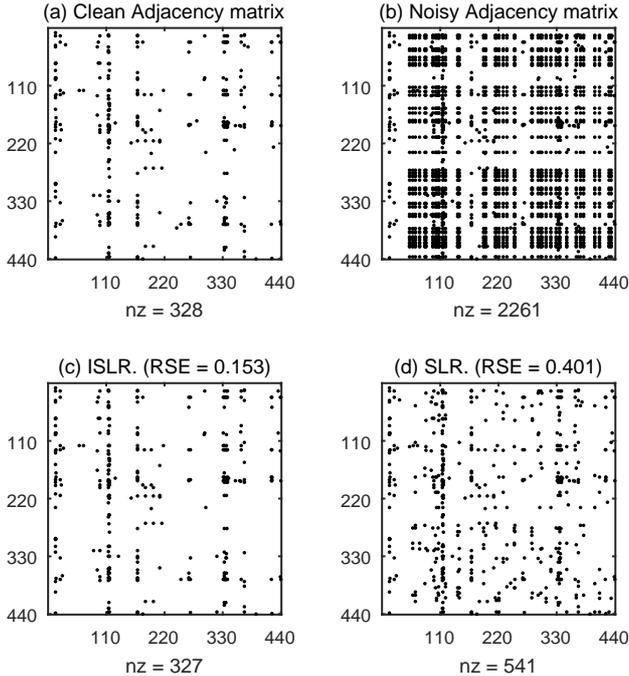}
	\caption{Illustration of denoising the weighted adjacency matrix representing protein interactions in the `Escherichia coli' bacteria. Note that `nz' represents the number of non-zero elements in the matrix.}
	\label{fig::protein example}
\end{figure}

Figure~\ref{fig::protein example}(a) shows the protein interaction data as a weighted adjacency matrix. The adjacency matrix is obtained after retaining 440 proteins of the entire set of 4394 proteins. We corrupt 10\% of the entries of the clean adjacency matrix, selected uniformly at random, with uniform noise in the interval $[0,\sigma]$. The noisy adjacency matrix is shown in Fig.~\ref{fig::protein example}(b), with $\sigma = 0.3$. We set the parameters $\lambda_i$ and $a_i, i = 1,2.$ as in the previous example.  Figure~\ref{fig::protein example}(c) and Fig.~\ref{fig::protein example}(d) show the denoised adjacency matrices obtained using the ISLR and the SLR methods, respectively. As in the case of previous example, the ISLR method offers a better RSE, and tends to correctly estimate the sparsity-pattern of the true matrix. 

\begin{figure}[t!]
	\centering
	\includegraphics[]{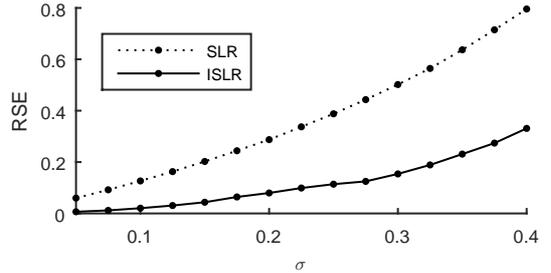}
	\caption{Average RSE as a function of $\sigma$.}
	\label{fig::RSE as a func of sigma}
\end{figure}

In order to assess the relative performance of the proposed ISLR method \eqref{eq::cost function} in comparison to the SLR method \eqref{eq::SLR}, we realize 15 noisy adjacency matrices and denoise them. For each value of $\sigma$, we choose the parameters $\lambda_i$, for both the methods, that yields the lowest RSE. Shown in Fig.~\ref{fig::RSE as a func of sigma} are the average RSE values as a function of $\sigma$. It can be seen that the ISLR method consistently offers a lower RSE.

\section{Conclusion}

We consider the problem of estimating a sparse low-rank matrix from its noisy observation. We generalize the convex formulation proposed for estimation of sparse low-rank matrix estimation \cite{Richard2012}, by utilizing non-convex sparsity-inducing regularizers. The non-convex penalty functions proposed are known to estimate the non-zero signal values more accurately. We show how to set the parameters of the non-convex penalty functions, so as to preserve the convexity of the overall problem (sum of data-fidelity and the rssegularization terms). The critical value of the non-convex penalty parameters define a convexity triangle; for all values of the nonconvex penalty parameters within this triangular region, the cost function is guaranteed to be strictly convex. 

We derive an efficient algorithm using ADMM with a single variable-splitting which solves the proposed convex objective function consisting of non-convex regularizers. We guarantee the convergence of ADMM to the global minimum of the objective function, provided the scalar augmented Lagrangian parameter $\mu$ is chosen such that $\mu > 1$. We illustrate several examples to demonstrate the effectiveness of the proposed formulation for estimation of sparse low-rank matrices.

The proposed method utilizes separable penalty functions (nonconvex) to induce sparsity stronger than separable convex penalty functions. A possible future direction involves the use of non-separable penalty functions, possibly nonconvex, that are designed so as to ensure the strict convexity of the objective function \cite{Selesnick2017}.

\section{Acknowledgements}

The authors thank the anonymous reviewers for their detailed suggestions and corrections. This work was supported by the ONR under grant N00014-15-1-2314 and the NSF under grant CCF-1525398.

\section{References}

\end{document}